\documentclass[12pt]{article}

\usepackage{amsmath}
\usepackage{amssymb}
\usepackage{amsthm}
\usepackage{exscale}
\usepackage[mathscr]{eucal}
\usepackage{float}
\usepackage{bm}
\usepackage[colorlinks, linkcolor=blue]{hyperref}

\newtheorem*{theoremA}{Theorem A}
\newtheorem*{theoremB}{Theorem B}
\newtheorem*{theoremC}{Theorem C}
\newtheorem*{theoremD}{Theorem D}

\newtheorem{theorem}{Theorem}[section]
\newtheorem{lemma}[theorem]{Lemma}
\newtheorem{corollary}[theorem]{Corollary}
\newtheorem{proposition}[theorem]{Proposition}

\newtheorem{remark}[theorem]{Remark}
\newtheorem{example}[theorem]{Example}

\newtheorem{definition}[theorem]{Definition}

\numberwithin{equation}{section}

\begin{document}

\title{On specification and measure expansiveness}
\author{Welington Cordeiro$^{a,b}$, Manfred Denker$^{b}$ and Xuan Zhang$^{b}$
\\ (a) Universidade Federal do Rio de Janeiro \\ 
and\\
(b)The  Pennsylvania State University}

\date{Aug 29, 2017}

\maketitle

\begin{abstract}  We relate the local specification and periodic shadowing properties. We also clarify the relation between local weak specification and local specification if the system is measure expansive.  The notion of strong measure expansiveness is introduced, and an example of a non-expansive systems with the strong measure expansive property is given.  Moreover, we find a family of examples with the $N$-expansive property, which are not strong measure expansive. We finally show a spectral decomposition theorem for strong measure expansive dynamical systems with shadowing.   
\end{abstract}

        
 \section{Introduction}\label{sec:intro} The theory of hyperbolic dynamical systems relies heavily on Anosov's closing lemma (\cite{An}) and the notion of specification as termed in Bowen's work in 1971 (\cite{Bo}). The abstract tracing property of finitely many  finite orbits by periodic points can be used to derive many interesting properties on the asymptotic behavior of such systems. In particular, Senti and two of the present authors (\cite{DSZ}) used such structures to prove fluctuation laws of ergodic sums over periodic orbits, which can be expressed in terms of  asymptotic normal distributions.

 The most common terminology originating from Anosov's closing lemma is, however, the notion of shadowing of pseudo-orbits. This raises the question of putting the notions used in \cite{DSZ} into this framework. Here we add the missing results which are not found in the literature. The common goal  of specification and shadowing  is to find a true trajectory near an approximate one, they only differ in understanding what an approximate trajectory should be. For specification one follows arbitrarily assembled finite pieces of orbits by a true orbit while for shadowing one traces a pseudo-orbit.

There are many generalizations of shadowing in the literature. Here we will consider limit shadowing and Lipschitz shadowing. Limit shadowing first appears in \cite{ENP}. From the numerical viewpoint it means that if we apply a numerical method that approximates the system with ``improving accuracy'' so that one-step errors tend to zero as time goes to infinity, then the numerically calculated orbits tend to real ones. The concept of Lipschitz shadowing was introduced in \cite{P}. It exists in many systems, for example Anosov systems share this property.

Both notions can be formulated in terms of approximating orbits being periodic or not periodic. 
 While in the non-periodic case, we prove in Theorem A that shadowing and specification properties are equivalent, we found three different definitions of periodic shadowing in the literature (see \cite{COP}, \cite{K} and \cite{OPT}).  It turns out that for expansive transformations one derives the existence of periodic points leading to equivalent statements for periodic shadowing and specification. Hence we are lead to study the equivalence under weaker assumptions than expansiveness. In this paper this is done using measure expansiveness.

The first research which considered expansiveness in dynamical systems was by Utz \cite{U}. In \cite{M}, Morales presented the interesting notion of $N$-expansive systems. Recently this notion has gained some attention, see for example \cite{A}, \cite{APV}, \cite{CC}. Furthermore, in \cite{MS} Morales and Sirvent introduced the notion of measure expansive property.

In this paper we define the notion of strong measure expansive system and in Theorems B and C we prove that for a transitive and strong measure expansive homeomorphism, the notions of shadowing and specification are equivalent. We explore the difference between the notions of $N$-expansiveness and strong measure expansiveness in Examples \ref{myex} and \ref{artiguex}. Example \ref{myex} was constructed by Carvalho and the first author in \cite{CC}, which shows that a $2$-expansive homeomorphism is not necessarily strong measure expansive. On the other hand, we show  in Example \ref{artiguex} that strong measure expansiveness does not necessarily imply expansiveness. The example was first shown by Artigue in \cite{A} and was used to explain that for each $N\in\mathbb{N}$ there is an $(N+1)$-expansive $C^N$-diffeomorphism which is not $N$-expansive.

In the final part, we extend the classical spectral decomposition theorem (see \cite{S, B}) for strong measure expansive homeomorphisms with the shadowing property in Theorem D. 
In particular, the theorem applies to Example \ref{artiguex}. However, Example \ref{myex} does not have a spectral decomposition hence $2$-expansiveness is not sufficient for a spectral decomposition theorem.

The current paper is organized as follows. In Section $2$ we collect some basic definitions of shadowing and specification. In Sections $3$ we treat the non-periodic case, in particular we prove that local weak specification is equivalent to shadowing, and we find several equivalent properties for the generalized definitions of shadowing. We also discuss the relation between local specification and various definitions of periodic shadowing. In Section $4$ we introduce the notion of strong measure expansiveness, and give some sufficient conditions for the equivalence between shadowing, periodic shadowing and local specification. Finally, in Section $5$ we prove a Spectral Decomposition Theorem for strong measure expansive systems with shadowing. The main results are Theorems A, B, C and D.

\section{Shadowing, specification and expansiveness: definitions}\label{sec:definitions}

We fix a  topological dynamical systems $(X,f)$, where $X$ is a compact metric space (unless stated otherwise)  with metric $d$ and $f:X\to X$ is a continuous map. Thus $\mathbb N=\{0,1,2,\ldots\}$ acts on $X$ via $f$, or in case $f$ is a homeomorphism $\mathbb Z$ acts on $X$. In order to unify notation we use $T$ to denote either $\mathbb N$ or $\mathbb Z$, depending on the nature of the action. Throughout the paper $(X,f)$ stands for such a dynamical system. $x$ or $y$ with or without sub- or superscripts will always denote points in $X$. We write $\Omega(f)$ for the set of non-wandering points of $f$ and write $\mbox{Per}_n(f)$ for the set of periodic points of $f$ with period $n$, but not necessarily its prime period, and $\mbox{Per}(f)$ for the set of all periodic points of $f$.  We also let $B(x,\epsilon)$ denote the open ball with center $x$ of radius $\epsilon>0$.

Recall the definition of \emph{stable} and \emph{unstable} sets for homeomorphisms $f$. For $x\in X$ and $\epsilon>0$ define
\begin{eqnarray*}
W^s_\epsilon(x,f)&=&\{y\in X;d(f^i(x),f^i(y))\le \epsilon, \ \forall i\geq 0\},\\
W^u_\epsilon(x,f)&=&\{y\in X;d(f^i(x),f^i(y))\le\epsilon, \ \forall i\leq 0\},\\
W^s(x,f)&=&\{y\in X;d(f^i(x),f^i(y))\rightarrow 0, \ i\rightarrow +\infty \},\\
W^u(x,f)&=&\{y\in X;d(f^i(x),f^i(y))\rightarrow 0, \ i\rightarrow -\infty \}.
\end{eqnarray*}
Also, for $x\in X$ and $\delta>0$ define $\Gamma_\delta(x,f)=W^s_\delta(x,f)\cap W^u_\delta(x,f)$. Note that $\Gamma_\delta(x,f)$ is a closed set. We will omit $f$ in the definitions when it is clear from the context.

For completeness we recall the definitions of \emph{shadowing} and related notions.

\begin{definition}\label{def:pseudo} {\rm [Pseudo orbit]}  A sequence $\{x_i\}_{i\in T}$ is called a $\delta$-\emph{pseudo orbit} if $d(f(x_i),x_{i+1})<\delta$ for each $i\in T$. A $\delta$-pseudo orbit $\{x_i\}_{i\in T}$ is \emph{periodic} if there is $n\in\mathbb{N}$ such that $x_i=x_{i+n}$ for each $i\in T$.  We say that $x\in X$ $\epsilon$-\emph{shadows} a $\delta$-pseudo orbit $\{x_i\}_{i\in T}$ if $d(f^{i}(x),x_i)<\epsilon$ for each $i\in T$.
\end{definition}

\begin{definition}\label{def:shadowing}{\rm [Shadowing properties] }  \newline
1. $f$ has the \emph{shadowing property} if for all $\epsilon>0$ there is $\delta>0$ such that if $\{x_i\}_{i\in T}$ is a $\delta$-pseudo orbit, then there is $x\in X$ which $\epsilon$-shadows $\{x_i\}_{i\in T}$.\newline
2. $f$ has the \emph{strong periodic shadowing} property if for all $\epsilon>0$ there is $\delta>0$ such that if $\{x_i\}_{i\in T}$ is a periodic $\delta$-pseudo orbit with period $N$, then there is $x\in \mbox{Per}_{N}(f)$ which $\epsilon$-shadows $\{x_i\}_{i\in T}$.
\newline
3.  (see \cite{OPT}) $f$ has the \emph{periodic shadowing} property if for all $\epsilon>0$ there is $\delta>0$ such that if $\{x_i\}_{i\in T}$ is a periodic $\delta$-pseudo orbit, then there is $x\in \mbox{Per}(f)$ which $\epsilon$-shadows $\{x_i\}_{i\in T}$.
\newline
4.  $f$ has the \emph{special periodic shadowing} property if it has the shadowing property and the periodic shadowing property.
\end{definition}

\begin{remark} In literature some authors use the term `periodic shadowing' to denote  `special periodic shadowing' as introduced in the previous definition, see for example \cite{K}. We use the above definitions to distinguish the different notions of periodic shadowing used in the literature.
\end{remark}

\begin{definition}\label{def:LimShadowing}{\rm [Generalized shadowing properties] } \newline
1.  (\cite{CK} and \cite{Sakai})  $f$ has the \emph{limit shadowing} property if for each sequence $\{x_i\}_{i\in\mathbb{N}}$ such that $\lim_{i\rightarrow\infty}d(f(x_i),x_{i+1})=0$,  there is $x\in X$ such that $$\lim_{i\rightarrow\infty}d(f^i(x),x_i)=0.$$
\newline
2.  (\cite{PT} and \cite{Sakai}). $f$ has the \emph{Lipschitz shadowing} property if there are constants $L,d_0>0$ such that if $\{x_i\}_{i\in T}$ is a sequence satisfying $d(f(x_i),x_{i+1})<d\leq d_0$, for every $i\in T$, then there is $x\in X$ such that $d(f^i(x),x_i)<Ld$, for every $i\in T$.
\newline
3. (\cite{CK})  A homeomorphism $f$ has the \emph{two-sided limit shadowing} property if for each sequence $\{x_i\}_{i\in\mathbb{Z}}$ such that $\lim_{i\rightarrow\pm\infty}d(f(x_i),x_{i+1})=0$, then there is $x\in X$ such that $\lim_{i\rightarrow\pm\infty}d(f^i(x),x_i)=0$.
\end{definition}

\begin{definition}\label{def:locspec} {\rm [Specification properties] } \newline
1. (\cite{Bo})  $f$ has the \emph{local specification} property if for all $\epsilon>0$ there are $N\in\mathbb{N}$ and $\delta>0$ such that if $k\in \mathbb N$ and  $x_1,...,x_k\in X$ satisfy
$d(f^n(x_i),x_{i+1})<\delta,$
with $N\leq n$, and $x_{k+1}=x_1$, then there is $x\in Per_{kn}(f)$ such that
$$d(f^{j}(x),f^{j-in}x_{i+1}) <\epsilon,$$
for all  $i\in\{0,1,..., k-1\} $  and   $in\leq j<(i+1)n.$\newline
2.  $f$ has the \emph{local weak specification} property if for all $\epsilon>0$ there are $N\in\mathbb{N}$ and $\delta>0$ such that if $k\in \mathbb N$ and $x_1,...,x_k\in X$ satisfy
$d(f^n(x_i),x_{i+1})<\delta,$
with $N\leq n$, then there is $x\in X$ such that
$$d(f^{j}(x),f^{j-in}x_{i+1})<\epsilon,$$
for all $ i\in\{0,1,..., k-1\} \text{ and }\ in\leq j<(i+1)n.$
\end{definition}

\begin{definition}\label{def:LimSpecification}{\rm [Generalized specification properties]}
\newline
1.  $f$ has the \emph{local  limit specification} property if there is $N\in\mathbb{N}$ such that for each sequence $\{x_i\}_{i\in\mathbb{N}}$ with $\lim_{i\rightarrow\infty}d(f^n(x_i),x_{i+1})=0$ for some $n\geq N$, there is $x\in X$ such that $$\lim_{i\rightarrow\infty}\sup_{ 0\leq j\leq n}d(f^{j+in}(x),f^j(x_i))=0.$$
\newline
2.   $f$ has the \emph{local Lipschitz specification} property if there are constants $L,d_0>0$ such that, for each $0<d\leq d_0$ there is $N\in\mathbb{N}$ such that if $\{x_i\}_{i\in T}$ is a sequence such that $d(f^n(x_i),x_{i+1})<d\leq d_0,$ for every $i\in T$, for some $n\geq N$, then there is $x\in X$ such that $$d(f^{in+j}(x),f^j(x_i))<Ld$$
 for every $i\in T$  and $0\le j< n$.
 \newline
 3. A homeomorphism $f$ has the \emph{two-sided limit specification} property if there is $N\in\mathbb{N}$ such that, for each sequence $\{x_i\}_{i\in\mathbb{Z}}$ such that $$\lim_{i\rightarrow\pm\infty}d(f^n(x_i),x_{i+1})=0$$ for some $n\geq N$, then there is $x\in X$ such that $$\lim_{i\rightarrow\pm\infty}\sup_{0<j\leq n}d(f^{i+j}(x),f^j(x_i))=0.$$
\end{definition}

Note that the specification here is in fact a weak specification as in Definition \ref{def:locspec}. The corresponding definition using periodic points obviously does not make much sense.

\begin{definition}{\rm [Measure expansive]}
\newline
1.  (\cite{MS})  A homeomorphims $f$ is called \emph{measure expansive} if there is $\delta>0$ such that for every invariant non-atomic probability  measure $\mu$ and $x\in{X}$ we have $\mu(\Gamma_\delta(x))=0$.
\newline
2. A homeomorphims $f$ is called \emph{strongly measure expansive} if there is $\delta>0$ such that for every invariant probability measure $\mu$ and $x\in X$ we have $\mu(\Gamma_\delta(x))=\mu({x})$.
\end{definition}

\begin{definition}{\rm [$n$-expansive] (\cite{M})} Let $n\in\mathbb{N}$. A homeomorphism $f$ is called \emph{$n$-expansive} if there is $\delta>0$ such that,  for every $x\in X$, $\Gamma_\delta(x)$ has at most $n$ elements.
\end{definition}

\section{Specification and shadowing}\label{sec:specshad}

\subsection{Local weak specification and shadowing properties}\label{sec:locspec}
In this subsection we  prove that local weak specification is equivalent to shadowing.

\begin{proposition}\label{prop:3.2} $f$ has the local weak specification property if, and only if, for all $\epsilon>0$ there are $N\in \mathbb N$ and $\delta>0$ such that if $\{x_i\}_{i\in T}$ satisfies
$d(f^n(x_i),x_{i+1})<\delta,$
with $N\leq n$, then there is $x\in X$ such that
$$d(f^{j+in}(x),f^{j}(x_i))<\epsilon$$
for every $i\in T$ and $0\le j<n$.
\end{proposition}
\begin{proof} Let $\epsilon>0$ be given and take $N\in\mathbb{N}$ and $\delta>0$ given by the local weak specification property. If $\{x_i\}_{i\in T}$ satisfies
$d(f^n(x_i),x_{i+1})<\delta,$
then, for each $k\in\mathbb{N}$, there is $x^k$ such that
$$d(f^{j+ in}(x^k),f^{j}(x_i))<\epsilon, \qquad 0\le |i| < k.$$
Since $X$ is compact and $f$ is continuous, any accumulation point $x$ of $\{x^k: k\ge 1\}$ satisfies
$$d(f^{j+in}(x),f^{j}(x_i))<\epsilon, \qquad i\in  T.$$
The proof is complete.
\end{proof}

\begin{theorem}\label{equiv}  $f$ has the local weak specification property if, and only if, $f$ has the shadowing property.
\end{theorem}

\begin{proof} By definition shadowing implies local weak specification with $N=1$ for every $\epsilon>0$.

Assume now that $f$ has the local weak specification property.  Let $\epsilon>0$ be given and choose $N\in\mathbb{N}$ and $\frac{\epsilon}{2}>\delta>0$ as in the  local weak specification property for $\epsilon/2$. Let $\eta>0$ be such that $$N\eta<\delta.$$
Pick $\delta_1>0$ such that if $d(x,y)<\delta_1$, then $d(f^i(x),f^i(y))<\eta$, for every $0\leq |i|\leq N$. Let $\{x_i\}_{i\in T}$ such that $d(f(x_i),x_{i+1})<\delta_1$. Define the sequence $y_i=x_{Ni}$ for $i\in T$. Then
\begin{eqnarray*} d(f^N(y_i),y_{i+1})&=&d(f^N(x_{Ni}),x_{N(i+1)}) \\
&\leq&\Sigma_{j=0}^{N-1} d(f^{N-j}(x_{Ni+j}),f^{N-j-1}(x_{Ni+j+1})) \\
&\leq&N\eta<\delta.
\end{eqnarray*}
By local weak specification there is $x\in X$ such that
$$d(f^{j}(x),f^{j-iN}(y_i))<\frac{\epsilon}{2}, \qquad i\in\{0,1,..\} \text{ and } iN\leq j<(i+1)N.$$
Therefore, if $i\in\{0,1,..\}$ and $iN\leq j<(i+1)N$,
\begin{eqnarray*} d(f^{j}(x),x_j)&\leq& d(f^{j}(x),f^{j-iN}(x_{Ni}))+\Sigma_{r=iN}^{j-1}d(f^{j-r}(x_r),f^{j-r-1}(x_{r+1})) \\
 &\leq&d(f^{j}(x),f^{j-iN}(y_{i}))+(j-Ni)\eta \\
&\leq&\frac{\epsilon}{2}+N\eta<\epsilon.
\end{eqnarray*}
This concludes the proof.
\end{proof}

\begin{theorem}\label{equiv2} Let $f:X\rightarrow X$ be a uniformly continuous function where $X$ is a locally compact metric space. Then $f$ has the local weak specification property if, and only if, $f$ has the shadowing property.
\end{theorem}

The proof of this theorem is similar to the previous theorem and we leave the proof to the reader.

\subsection{Generalized shadowing and local weak specification}\label{sec:speci}

There are several generalizations of the notion of shadowing as explained in Section \ref{sec:definitions}. Here we shall discuss three of them  and their corresponding versions for local specification. When $f$  is a homeomorphism, we call it Lipschitz if $f$ and $f^{-1}$ are both Lipschitz.

\begin{theoremA}\label{A}
\begin{enumerate}
\item $f$ has the limit shadowing property if, and only if, $f$ has the local limit specification property.
\item If $f:X\rightarrow X$ is a Lipschitz map, then $f$ has the Lipschtiz shadowing property if, and only if, $f$ has the local Lipschitz specification property.
\item $f$ has the two-sided limit shadowing property if, and only if, $f$ has the local two-sided limit specification property.
\end{enumerate}
\end{theoremA}

The proof of the Theorem $A$ follows from the next three lemmas.

\begin{lemma}\label{lem:3.4} Let $f:X\rightarrow X$ be a uniformly continuous map where $X$ is a locally compact metric space. Then $f$ has the limit shadowing property if, and only if, $f$ has the local limit specification property.
\end{lemma}
\begin{proof} By definition limit shadowing implies local limit specification with $N=1$.

Assume $f$ has the local limit specification and let $N>0$ be given as in Definition \ref{def:LimSpecification}. Let $\{x_i\}_{i\in\mathbb{N}}$ be such that $\lim_{i\rightarrow\infty}d(f(x_i),x_{i+1})=0$. Define $y_i=x_{Ni}$ for each $i\in\mathbb{N}$. We will prove that $\lim_{i\rightarrow\infty}d(f^N(y_i),y_{i+1})=0$. Let $\epsilon>0$ be given, and take $\delta>0$ such that if $d(x,y)<\delta$, then $d(f^i(x),f^i(y))<\epsilon/N$, for each $0\leq i\leq N$. Since $$\lim_{i\rightarrow\infty}d(f(x_i),x_{i+1})=0$$ there is $M>0$ such that if $i>M$ then $d(f(x_i),x_{i+1})<\delta$. Therefore, if $iN>M$ then
\begin{eqnarray*} d(f^N(y_i),y_{i+1})&=&d(f^N(x_{iN}),x_{(i+1)N}) \\
&\leq&\Sigma_{j=0}^{N-1} d(f^{N-j}(x_{iN+j}),f^{N-(j+1)}(x_{iN+j+1})) \\
&\leq& N\frac{\epsilon}{N}=\epsilon,
\end{eqnarray*}
because $$d(f(x_{iN+j}),x_{iN+j+1})<\delta,$$ for every $0\leq j<N$. Hence $\lim_{i\rightarrow\infty}d(f^N(y_i),y_{i+1})=0$, and by local limit specification property there is $x\in X$ such that $$\lim_{i\rightarrow\infty}\sup_{ 0<j\leq N}d(f^{j+iN}(x),f^j(y_i))=0.$$

Now we will prove $\lim_{i\rightarrow\infty}d(f^i(x),x_i)=0$. Indeed, for $\epsilon>0$ let $\delta>0$ be such that if $d(x,y)<\delta$ then $d(f^i(x),f^i(y))<\epsilon/2N$. Let $M>0$ be such that $d(f(x_i),x_{i+1})<\delta$ and $$d(f^{j+iN}(x),f^j(y_i))<\epsilon/2$$ if $i\geq M$ and $0\leq j<N$. If $i\geq M$ we can find $j>0$ and $0\leq l<N$ such that $jN+l=i$. Thus
\begin{eqnarray*} d(f^i(x),x_i)&=&d(f^{jN+l}(x),x_{jN+l}) \\
&\leq&d(f^{jN+l}(x),f^l(y_j))+d(f^l(y_j),x_{jN+l}) \\
&<&\frac{\epsilon}{2}+d(f^l(x_{Nj}),x_{jN+l}) \\
&\leq&\frac{\epsilon}{2}+\Sigma_{q=0}^{l-1} d(f^{l-q}(x_{Nj+q}),f^{l-(q+1)}(x_{Nj+q+1})) \\
&<&\frac{\epsilon}{2}+l\frac{\epsilon}{2N}<\epsilon,
\end{eqnarray*}
because $d(f(x_{Nj+q}),x_{Nj+q+1})<\delta$. Hence $$\lim_{i\rightarrow\infty}d(f^i(x),x_i)=0.$$

Thus $f$ has the limit shadowing property. \end{proof}

\begin{lemma} Let $f:X\rightarrow X$ be a Lipschitz map. Then $f$ has the Lipschitz shadowing property if, and only if, $f$ has the local Lipschitz specification property.
\end{lemma}
\begin{proof} By definition Lipschitz shadowing implies local Lipschitz specification with $N=1$.

Assume $f$ has the local Lipschitz specification and let $L>0$ be greater than the Lipschitz constant of the map and of the local Lipschitz specification property.
We will show that $f$ has the Lipschitz shadowing property with constants $L'=(L+1)N^2L$ and $d_0'=d_0/N^2L$, where $N$ is chosen below.
Let $0<d<d_0'$ be given and fix  $N\in\mathbb{N}$ as given by the local Lipschitz specification property.
Let $\{x_i\}_{i\in T}$ be such that $d(f(x_i),x_{i+1})<d$. Define the sequence $y_i=x_{Ni}$ for $i\in T$. Then
\begin{eqnarray*} d(f^N(y_i),y_{i+1})&=&d(f^N(x_{Ni}),x_{N(i+1)}) \\
&\leq&\Sigma_{j=0}^{N-1} d(f^{N-j}(x_{Ni+j}),f^{N-j-1}(x_{Ni+j+1})) \\
&\leq&N^2Ld<d_0.
\end{eqnarray*}
By the local Lipschitz specification property there is $x\in X$ such that
$$d(f^{j+iN}(x),f^j(y_i))<LN^2Ld=(LN)^2d,$$ for every $i\in T$, and $0\le j< N$.
Therefore, if $i\in T$ and $iN\leq j<(i+1)N$,
\begin{eqnarray*} d(f^{j}(x),x_j)&\leq& d(f^{j}(x),f^{j-iN}(x_{Ni}))+\Sigma_{r=iN}^{j-1}d(f^{j-r}(x_r),f^{j-r-1}(x_{r+1})) \\
 &\leq&d(f^{j}(x),f^{j-iN}(y_{i}))+(j-Ni)^2Ld \\
&\leq&(LN)^2d+N^2Ld<(L+1)N^2Ld=L'd.
\end{eqnarray*}
This concludes the proof.  \end{proof}

\begin{lemma} Let $f:X\rightarrow X$ be a homeomorphism. Then $f$ has the two-sided limit shadowing property if, and only if, $f$ has the local two-sided specification property.
\end{lemma}

\begin{proof}  This follows from a similar proof to Lemma \ref{lem:3.4}. \end{proof}

\subsection{Periodic shadowing and local specification}
In this section we shall show some relationships between periodic shadowing and local specification properties.

\begin{theoremB}\label{B} If $f$ has the local specification property, then $f$ has the periodic shadowing property.
\end{theoremB}
\begin{proof}
Assume $f$ has the local specification property. Let $\epsilon>0$ be given and take $N\in\mathbb{N}$ and $\frac{\epsilon}{2}>\delta>0$ by the local specification property applied to $\epsilon/2$. Let $\eta>0$ be such that
$$N\eta<\delta.$$
Get $\delta_1>0$ such that if $d(x,y)<\delta_1$, then $d(f^i(x),f^i(y))<\eta$ for every $-N\leq i\leq N$. Let $\{x_i\}_{i\in T}$ such that $d(f(x_i),x_{i+1})<\delta_1$, and $x_{i+M}=x_i$ for each $i\in T$. Define the sequence $y_i=x_{Ni}$ for $i\in T$.
Then
\begin{eqnarray*} d(f^N(y_i),y_{i+1})&=&d(f^N(x_{Ni}),x_{N(i+1)}) \\
&\leq&\Sigma_{j=0}^{N-1} d(f^{N-j}(x_{Ni+j}),f^{N-j-1}(x_{Ni+j+1})) \\
&\leq&N\eta<\delta,
\end{eqnarray*}
and therefore,
\begin{eqnarray*} d(f^{N}(y_{M-1}),y_{1})&=&d(f^{N}(x_{N(M-1)}),x_{N}) \\
&=&d(f^{N}(x_{N(M-1)}),x_{NM}) \\
&=&d(f^{N}(y_{M-1}),y_{M})<\delta.
\end{eqnarray*}
By the local specification property there is $x\in Per(f)$ such that
$$d(f^{j}(x),f^{j-iN}(y_i))<\frac{\epsilon}{2}, \text{ for } i\in\{0,1,..\} \text{ and } iN\leq j<(i+1)N.$$
Therefore, if $i\in\{0,1,..\}$ and $iN\leq j<(i+1)N$,
\begin{eqnarray*} d(f^{j}(x),x_j)&\leq& d(f^{j}(x),f^{j-iN}(x_{Ni}))+\Sigma_{r=iN}^{j-1}d(f^{j-r}(x_r),f^{j-r-1}(x_{r+1})) \\
 &\leq&d(f^{j}(x),f^{j-iN}(y_{i}))+(j-Ni)\eta \\
&\leq&\frac{\epsilon}{2}+N\eta<\epsilon,
\end{eqnarray*}
which concludes the proof.
\end{proof}

\noindent\textbf{Remark:} If $f$ has the strong periodic shadowing property, then $f$ has the local specification property with $N=1$ for every $\epsilon>0$. Therefore,
\newline
\indent \textit{Strong Periodic Shadowing $\Rightarrow$ Local Specification $\Rightarrow$ Periodic Shadowing.}

\begin{corollary}\label{prop:3.3} Let $f$ be an expansive homeomorphism. If $f$ has the local weak specification, then $f$ has the local specification.
\end{corollary}
\begin{proof} By Theorem \ref{equiv} if $f$ has the local weak specification, then $f$ has the shadowing property. It is well known (see \cite{Ao-Hi}, for example)  that if $f$ is an expansive homeomorphism with the shadowing property, then $f$ has the strong periodic shadowing property.
Therefore, $f$ has the local specification property.
\end{proof}

We recall that a homeomorphism $f:X\rightarrow X$ is \emph{transitive} if there is a point $x$ with a dense orbit $\{f^n(x);n\in\mathbb Z\}$.

\begin{proposition}\label{porp:3.1} Let $f$ be a transitive homeomorphism with the periodic shadowing property, then $f$ has the local weak specification property.
\end{proposition}
\begin{proof}
We shall prove that $f$ has the finite shadowing property, i.e., for every $\epsilon>0$ there is $\delta>0$ such that if $\{x_i\}_{i=0}^n$, with $d(f(x_i),x_{i+1})<\delta$,
then there is $x\in X$ such that $d(f^i(x),x_i)<\delta$ for $0\leq i\leq n$. Let $\epsilon$ be given. Take $\delta>0$ as in the periodic shadowing property for $\epsilon/2$. If $\{x_i\}_{i=0}^n$ is such that $d(f(x_i),x_{i+1})<\delta$, by transitivity, there is $N\in\mathbb{N}$ and $a\in B_{\delta}(f(x_n))\cap f^{-N}(B_{\delta}(x_0))$. Define $\{y_i\}_{i\in\mathbb{Z}}$, by $y_i=x_i$ if $0\leq i\leq n$, $y_i=f^{i-n}(a)$ if $n<i\leq n+N-1$ and with period $n+N$. Then $\{y_i\}$ is a periodic $\delta$-pseudo orbit, and by periodic shadowing there is $y\in Per(f)$ such that  $d(f^i(y),y_i)<\delta$ for each $i\in\mathbb{Z}$. Therefore, for each $0\leq i\leq n$,
$$d(f^i(y),x_i)=d(f^i(y),y_i)<\delta.$$
Thus $f$ has the finite periodic shadowing. Notice that if $f$ has the finite shadowing property, then $f$ has the shadowing property. In fact, if $\{x_i\}_{i\in\mathbb{Z}}$ is a $\delta$-pseudo orbit, then for each finite $\delta$-pseudo orbit $\{x_i\}_{i=-N}^N$ we can find a $\frac{\epsilon}{2}$-shadowing point $z_N$ for it, and if $z$ is a accumulation point of $\{z_N\}_{N\in\mathbb{N}}$, then $z$ $\epsilon$-shadows the pseudo  orbit  $\{x_i\}$. Hence $f$ has the shadowing property. By Theorem \ref{equiv} $f$ has the local weak specification property.
\end{proof}

\begin{corollary}\label{lem:3.1} Let $f$ be a transitive homeomorphism. Then $f$ has the periodic shadowing property, if and only if, $f$ has the special shadowing property.
\end{corollary}

\section{Measure expansive systems}
In this section we study strong measure expansive homeomorphisms and show that strong measure expansive homemomorphims with shadowing have the strong periodic shadowing property. On one hand we show by an example that a strong measure expansive homeomorphism does not need to be  expansive. On the other hand, we give an example of a measure expansive homeomorphism, which is not strong measure expansive. We summarize the results in Theorem C.

\begin{theorem}\label{measure1} Let $f$ be a measure expansive homeomorphism. If $f$ has the shadowing property, then $f$ has the periodic shadowing property.
\end{theorem}
\begin{proof} Let $\epsilon_0>0$ be the measure expansive constant of $f$, $\epsilon_0>2\epsilon$ and $\delta_2$ be given by the shadowing property for $\epsilon/2>0$. Assume $\{x_i\}_{i\in\mathbb{Z}}$ is a periodic $\delta$-pseudo orbit with period $N$ and $\delta<\delta_2$. Then there is $x\in X$ such that $d(f^i(x),x_i)<\epsilon/2$. Then we have two cases: $x$ is periodic or not. If $x$ is periodic nothing has to be shown.  Assume now that $x$ is not periodic. Fix $0\leq l<N$ and $k\in\mathbb{Z}$. Since $x_{i+l+kN}=x_{i+l}$
\begin{eqnarray*} d(f^i(f^l(x)),f^i(f^{kN}(f^l(x))))&\leq&d(f^i(f^l(x)),x_{i+l})+d(x_{i+l},f^i(f^{kN}(f^l(x)))) \\
&=& d(f^{i+l}(x),x_{i+l})+d(x_{i+l+kN},f^{i+kN+l}(x)) \\
&<& \frac{\epsilon}{2}+\frac{\epsilon}{2}=\epsilon.
\end{eqnarray*}
Therefore, $f^{kN}(f^l(x))\in\Gamma_\epsilon(f^l(x))$ for every $k\in\mathbb{Z}$. Thus the closure of the orbit of $x$ is a subset of $\bigcup_{l=0}^{N-1}\Gamma_{\epsilon}(f^l(x))$. Take $\mu$ to be a weak accumulation point of the uniform distributions supported on the finite orbits $x, f(x),...,f^{N-1}(x)$. Then $\mu$ is an invariant measure. Therefore $$\mu(\bigcup_{l=0}^{N-1}\Gamma_{\epsilon}(f^l(x)))=1,$$ and $\mu(\Gamma_{\epsilon}(f^l(x)))>0$ for some $0\leq l<N$. Since $f$ is measure expansive with constant $\epsilon$ it follows that $\mu$ is atomic. Thus there is $z\in\bigcup_{l=0}^{N-1}\Gamma_{\epsilon}(f^l(x))$ such that $\mu(z)>0$. Since $\mu$ is an invariant probability measure, $z$ is periodic. Hence $z\in\Gamma_\epsilon(f^l(x))$, for some $0\leq l<N$, and
\begin{eqnarray*} d(x_i,f^i(f^{-l}(z)))&\leq& d(x_i,f^i(x))+d(f^i(x),f^i(f^{-l}(z))) \\
&=& d(x_i,f^i(x))+d(f^i(x),f^{i-l}(z)) \\
&\leq& \epsilon+\epsilon=2\epsilon.
\end{eqnarray*}
\end{proof}

\begin{corollary}\label{measure1.1} Let $f$ be an measure expansive homeomorphism. Then $f$ has the shadowing property if, and only if, $f$ has the special shadowing property.
\end{corollary}

By definitions, if a homeomorphism is strong measure expansive, then it is measure expansive as well.

\begin{lemma} Let $f:X\rightarrow X$ be a homeomorphism without periodic points. $f$ is measure expansive if, and only if, $f$ is strong measure expansive.
\end{lemma}
\begin{proof} By definitions, if $f$ is strong measure expansive, then it is measure expansive.
Assume $f$ is measure expansive and let $\delta>0$ be given by the measure expansive property. If $\mu$ is an invariant atomic measure, then $\mu$ is supported on periodic orbits, a contradiction since $f$ has no periodic points.
If $\mu$ is an invariant non-atomic probability measure, then $\mu(\Gamma_\delta(x))=0=\mu(x)$, for every $x\in X$. Hence $f$ is strong measure expansive.
\end{proof}

\begin{lemma}\label{measure2} Let $f:X\rightarrow X$ be a homeomorphism. If $f$ is strong measure expansive, then $f|Per(f)$ is expansive.
\end{lemma}
\begin{proof} Let $\delta>0$ the constant of strong measure expansiveness  of $f$.
If $x\neq y\in Per(f)$ satisfies $y\in\Gamma_\delta(x)$, let $m$ be the period of $x$ and $n$ be the period of $y$.
Let $\mu$ be the invariant probability measure such that for each $p\in O(x)\cup O(y)$ we get $\mu(p)=1/(m+n)$. Hence
$$\mu(\Gamma_\delta(x))\geq\mu(x)+\mu(y)=\frac{2}{m+n}>\frac{1}{m+n}=\mu(x).$$
Contradiction. Therefore, $f|Per(f)$ is expansive.
\end{proof}

By definition $1$-expansiveness is equivalent to expansiveness, and it is well known (see \cite{MS}) that for every $n\in\mathbb{N}$ if $f$ is $n$-expansive, then $f$ is measure expansive.

\begin{example}\label{myex}\cite[Theorem A]{CC} Let $A:\mathbb{T}^2\rightarrow\mathbb{T}^2$ be a linear Anosov diffeomorphism of the torus. We will create a new compact metric space $X$ and a homeomorphism $f:X\rightarrow X$ that is a measure expansive homeomorphism, but not strong measure expansive.
Further, suppose $\{p_k\}_{k\in\mathbb{N}}$ is a sequence of periodic points of $A$, which we suppose belong to different orbits. Define $X$ as the set $\mathbb{T}^2\cup E$ where $E$ is an infinite enumerable set.

For each $k\in\mathbb{N}$ let $\pi(p_k)$ denote the prime period of $p_k$ and for each $j\in\{0,\dots,\pi(p_k)-1\}$ consider a point $q(k,j)\in E$. We can choose these points pairwise different. Define a distance $D$ on $X$ by
$$D(x,y)=\begin{cases}d(x,y), &x,y\in \mathbb{T}^2,\\
\frac{1}{k}+d(y,A^j(p_k)), &x=q(k,j), y\in \mathbb{T}^2,\\
\frac{1}{k}+\frac{1}{m}+d(A^j(p_k),A^r(p_m)), &x=q(k,j), y=q(m,r), k\neq m~\text{or}~j\neq r.
\end{cases}$$
Define the homeomorphism $f:X\rightarrow X$ by $$f(x)=\begin{cases} A(x), &x\in \mathbb{T}^2,\\
q(k,(j+1)\hspace{-0.3cm}\mod\pi(p_k)), &x=q(k,j).
\end{cases}$$ Note that $E$ splits into an infinite number of periodic orbits of $f$. Indeed, for $k\in\mathbb{N}$, the set $\{q(k,j)\,;\,j\in\{0,\dots,\pi(p_k)-1\}\}$ is a periodic orbit for $f$ with period $\pi(p_k)$. By \cite{CC} $f:X\rightarrow X$ is a measure expansive homeomorphism on the compact metric $(X,D)$. But $f$ is not strong measure expansive, otherwise by Lemma \ref{measure2} $f|Per(f)$ must be expansive and since for every $\delta>0$ we can find $k>0$ large enough such that $q(k,0)\in\Gamma_\delta(p_k)$, we obtain a  contradiction since $q(k,0)\neq p_k$.
\end{example}

\begin{example}\label{artiguex} We will construct a family of strong measure expansive homeomorphism, which are $N+1$-expansive but not $N$-expansive. In \cite{A} Artigue  defines  a compact metric space $S$, and for each $N\in\mathbb{N}$, an $N+1$-expansive homeomorphism $f_N:S\rightarrow S$, which is not $N$-expansive, but $f_N|\Omega(f_N)$ is expansive and $\Omega(f_N)$ is isolated, i.e., there is an open set $U$ such that $\Omega(f_N)=\bigcap_{n\in\mathbb{Z}}f_N^n(U)$. We will show that $f_N$ is strong measure expansive.

Define $\delta=\min(\delta_1,\delta_2,\delta_3)$, where $\delta_1>0$ is an expansive constant of $f_N|\Omega(f_N)$, $\delta_2>0$ is an $N+1$-expansive constant of $f$ and $\delta_3>0$ is such that if $p\in\Omega(f_N)$, then $B_\delta(p)\subset U$. Let $\mu$ be an invariant probability measure on $S$, and assume that there is $x\in S$ such that $\mu(\Gamma_\delta(x))\neq\mu(\{x\})$. By the choice of $\delta<\delta_2$ we have that $\Gamma_\delta(x)$ is a finite set, and therefore there is a finite set $A$ such that $\mu(A)>0$, and $x\neq A$. Since $\mu$ is an invariant probability measure there is at least a periodic point $y$ such that $\mu(y)>0$, and thus $y\in\Omega(f_N)$. By the choice of $\delta<\delta_1$ we have that $x\notin\Omega(f_N)$. Since $$d(f_N^i(x),f_N^i(y))<\delta\leq\delta_3, \ \forall i\in\mathbb{Z},$$ we get $f_N^i(x)\in U$, for every $i\in\mathbb{Z}$, and hence $x\in\Omega(f_N)$. This contradiction implies that $f_N$ is strong measure expansive.
\end{example}

Beside the expansive properties under discussion, there are other kinds of expansive properties giving insights to the maps which are not expansive. For example, one can discuss countably-expansive property (\cite{MS, ACO}) if $\Gamma_\delta(x)$ can contain at most countably many points and discuss continuum-wise (cw) expansive property (\cite{Ka, ACO}) if the map ``expands'' every non-degenerate subcontinuum. We refer the readers to the cited papers for the exact definitions. In the table below we summarize what is known about expansive properties.
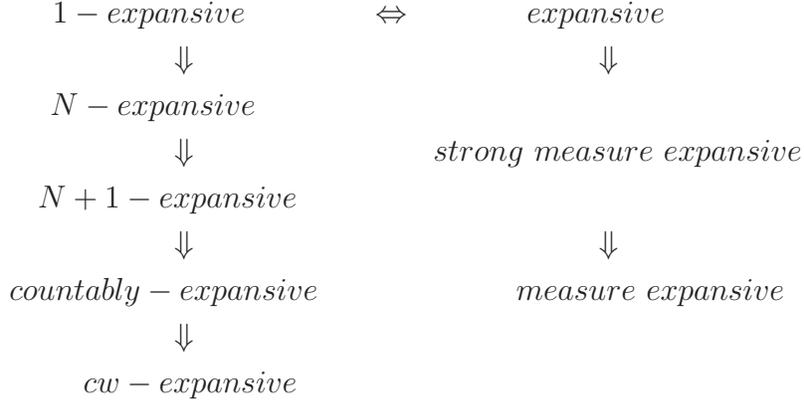
\begin{eqnarray*}  1-expansive \ \ \ \ \ \ \ \ \ \ &\Leftrightarrow& \ \ \ \ \ \ \ \ \ expansive  \\
 \Downarrow \ \ \ \ \ \ \ \ \ \ \ \ \ \ \  &  &  \ \ \ \ \ \ \ \ \ \ \ \ \ \ \ \  \Downarrow   \\
 N-expansive \ \ \ \ \ \ \ \ \ & &   \\
\Downarrow \ \ \ \ \ \ \ \ \ \ \ \ \ \ \ & & strong \ measure \ expansive  \\
N+1-expansive \ \ \ \ \ & & \ \ \ \ \ \ \ \ \ \ \ \ \ \ \ \  \\
 \Downarrow \ \ \ \ \ \ \ \ \ \ \ \ \ \ \ &  & \ \ \ \ \ \ \ \ \ \ \ \ \ \ \ \    \Downarrow \\
countably-expansive \ \ \ & & \ \  \ \ \ \ \ \ measure \ expansive \\
\Downarrow\ \ \ \ \ \ \ \ \ \ \ \ \ \ \ & & \\
cw-expansive\ \ \ \ \ &  &
\end{eqnarray*}
The previous examples show that there are $n$-expansive homeomorphisms with and without the strong  measure expansive  property.

\begin{theorem}\label{ShadStrShad} Let $f$ be a strong measure expansive homeomorphism with the shadowing property. Then $f$ has the strong periodic shadowing property.
\end{theorem}
\begin{proof} By Theorem \ref{measure1} $f$ has  periodic shadowing property. Let $\epsilon>0$ be the strong measure expansive constant of $f$ and $\delta>0$ given by the periodic shadowing property for $\epsilon/2$. Let $\{x_i\}_{i\in T}$ a periodic $\delta$-pseudo orbit with period $N$. By periodic weak shadowing there is $x\in Per(f)$ such that $$d(f^i(x),x_i)<\frac{\epsilon}{2}, \ \forall i\in T. $$
Therefore,
\begin{eqnarray*} d(f^i(x),f^i(f^N(x)))&\leq& d(f^i(x),x_i)+d(x_i,f^i(f^N(x))) \\
&=& d(f^i(x),x_i)+d(x_{i+N},f^{i+N}(x)) \\
&<&\frac{\epsilon}{2}+\frac{\epsilon}{2}=\epsilon.
\end{eqnarray*}
Since by Lemma \ref{measure2} $f|Per(f)$ is expansive with the expansive constant equal to $\epsilon$, we have $f^N(x)=x$. It finishes the proof.
\end{proof}

\begin{corollary} Let $f$ be a strong measure expansive homeomorphism. If $f$ has the local weak specification property, then $f$ has the local specification property.
\end{corollary}

\begin{example} (See \cite{CK}) Let $p$ and $q$ be relatively prime integers. Set $r=p+q-1$. Define a shift of finite type $X(p,q)\subset\Omega_r$ by specifying
\begin{eqnarray*}\Im = \{01,12,...,(p-2)(p-1),(p-1)0,0p,p(p+1),..., \\
(p+q-2)(p+q-1),(p+q-1)0\}.
\end{eqnarray*}
In the others words, $X_{(p,q)}$ consists of sequences of vertices visited during  a bi-infinite walk on the directed graph with two loops: one of length $p$ with vertices labeled $0,...,p-1$ and one of length $q$ with vertices labeled $0,p,p+1,...,p+q-2$. Since the graph is connected and has two cycles with relatively prime lengths it presents a topologically mixing shift of finite type. Moreover, this shift of finite type does not have any periodic point with prime period smaller than $\min\{p, q\}$. Let $(p_j)_{j=1}^\infty$ be a strictly increasing sequence of prime numbers. Let $X_n=X(p_n,p_{n-1})$ for $n\in\mathbb{N}$, and $\sigma_n$ be a shift transformation on $\Omega_{p_n+p_{n+1}-1}$ restricted to $X_n$. By \cite{CK} the product system $F=\sigma_1\times\sigma_2\times...$ on $X=\Pi_{i=1}^\infty X_n$ is topological mixing with the shadowing property, but it does not have the periodic shadowing property, since there are no periodic points.
\end{example}

\begin{theoremC}\label{cor:3.4} Let $f:X\rightarrow X$ be a transitive and strong measure expansive homeomorphism. Then the followings are equivalent:
\begin{enumerate}
\item $f$ has the local specification property;
\item $f$ has the local weak specification property;
\item $f$ has the shadowing property;
\item $f$ has the strong periodic shadowing property;
\item $f$ has the periodic shadowing property;
\item $f$ has the special shadowing property.
\end{enumerate}
 \end{theoremC}
\begin{proof} By Theorem B, $(1)$ implies $(5)$. Since $f$ is transitive, by Proposition \ref{porp:3.1}, $(5)$ implies $(2)$. By Theorem \ref{equiv}, $(2)$ implies $(3)$. Since $f$ is strong measure expansive $(3)$ implies $(4)$, by Theorem \ref{ShadStrShad}. By definition, $(1)$ implies $(5)$ (with $N=1$ for every $\epsilon>0$). Finally by Lemma \ref{lem:3.1}, $(5)$ is equivalent to $(6)$.
\end{proof}

\section{Spectral Theorem}

We recall the definition of the {chain recurrent set} of $f$. For $x,y\in X$ and $\delta>0$, $x\stackrel{\delta}{\rightleftharpoons} y$ if there are $\delta$-pseudo orbits such that $x=x_0,x_1,...,x_l=y$ and $y=y_0,y_1,...,y_r=x$. If $x\stackrel{\delta}{\rightleftharpoons} y$ for every $\delta>0$, then $x\rightleftharpoons y$. The set $$CR(f)=\{x\in X;x\rightleftharpoons x\}$$ is the \emph{chain recurrent set} of $f$. It is well known  that $CR(f)$ is closed and if $f$ has the shadowing property then $CR(f)=\Omega(f)$ (see \cite{Ao-Hi}, [Lemma 3.1.1 and Theorem 3.1.2]).

It is well known that if $f:X\rightarrow X$ is an expansive homeomorphism defined on a compact metric space then there is $\epsilon>0$ such that for each $x\in X$ the local stable set $W^s_\epsilon(x,f)$ is a subset of the stable set $W^s(x,f)$. By Lemma \ref{measure2} $f|Per(f)$ is expansive, but since $Per(f)$ in general is not closed, in principle, expansiveness of $f|Per(f)$ does not imply the existence of $\epsilon>0$ such that if $x\in Per(f)$, then $W^s_\epsilon(x,f|Per(f))\subset W^s(x,f|Per(f))$. Note that since $W^s_\epsilon(x,f|Per(f))\subset W^s_\epsilon(x,f)$ the next theorem shows  a stronger property for the existence of a such $\epsilon$ for $f|Per(f)$.

\begin{theorem}\label{stableset} Let $f$ be a strong measure expansive homeomorphism with shadowing. There is $\epsilon>0$ such that if $p\in Per(f)$, then $W^s_\epsilon(p,f)\subset W^s(p,f)$ and $W^u_\epsilon(p,f)\subset W^u(p,f)$.
\end{theorem}

\begin{proof}  Let $\epsilon=\delta/2$, where $\delta>0$ is the strong measure expansive constant of $f$. Assume that $y\in W^s_\epsilon(p)$, but $y\notin W^s(p)$ for some $p\in Per(f)$. Hence, there are a sequence of natural numbers $\{m_n\}_{n\in\mathbb{N}}$ and $r>0$ such that $m_n\rightarrow\infty$ and $$r<d(f^{m_n}(y),f^{m_n}(p))<\epsilon,$$
for each $n\in\mathbb{N}$. Using subsequences if necessary, we can assume that $f^{m_n}(y)\rightarrow y_0$ and $f^{m_n}(p)\rightarrow p_0$. Since $p$ is periodic, we have that $p_0=f^k(p)$ for some $k\in\mathbb N$. Then
$$d(p_0,y_0)=\lim_{n\rightarrow\infty} d(f^{m_n}(y),f^{m_n}(p))\geq r>0,$$
and $p_0\neq y_0$. We claim that $p_0\in\Gamma_\delta(y_0)$. Fix $i\in\mathbb{Z}$ and let $N\in\mathbb{N}$ be such that $m_N+i>0$. Since $y\in W^s_\epsilon(p)$,
\begin{eqnarray*} d(f^i(p_0),f^i(y_0))&=&\lim_{N\rightarrow\infty} d(f^i(f^{m_N}(y)),f^i(f^{m_N}(p))) \\
&=&\lim_{N\rightarrow\infty} d(f^{i+m_N}(y),f^{i+m_N}(p)) \\
&\leq& \epsilon=\delta/2<\delta.
\end{eqnarray*}
This concludes the proof of the claim. Let $\mu$ be the invariant probability measure supported on the orbit of $p_0$. Since $p_0$ is periodic $\mu(p_0)>0$. Hence,
$$\mu(\{y_0\})<\mu(\{y_0\})+\mu(\{p_0\})\leq\mu(\Gamma_\delta(y_0)).$$
But $\mu(\Gamma_\delta(y_0))=\mu(\{y_0\})$ by choice of $\delta$. This contradiction completes the proof.
\end{proof}

The proof of the next theorem is similar to the classical proof of the spectral decomposition theorem for an expansive homeomorphism. The main difference is that we employ the previous theorem when dealing with periodic points.


\begin{theoremD}\label{specdec} Let $f$ be a strong measure expansive homeomorphism with shadowing. Then,
\begin{enumerate}
\item $\Omega(f)=\bigcup_{i=1}^lB_i$, where each $B_i$ is invariant, closed and $f|B_i$ is transitive.
\item For each $0\leq k\leq l$ $B_k=\bigcup_{i=0}^{a_k-1}C_i$, where $C_i\cap C_j=\emptyset$ if $i\neq j$, and for each $0\leq i<a_k$, $C_i$ is closed, $f(C_i)=C_{i+1}$, $f^{a_k}(C_i)=C_i$ and $f^{a_k}|C_i$ is topological mixing.
\end{enumerate}
\end{theoremD}
\begin{proof}\footnote{We thank Ngoc-Thach Nguyen for pointing out a flaw in the original proof to us.} By \cite[Theorem 3.1.2]{Ao-Hi} $\Omega(f)=CR(f)$, therefore $\Omega(f)=\bigcup_{i\in\Lambda} B_i$, where $B_i$ are the equivalence classes of the relation $\rightleftharpoons$ on $CR(f)$. Let $\delta>0$ be given by the shadowing property for $\epsilon=\frac{1}{2}\epsilon_1>0$, where $\epsilon_1>0$ is given by Theorem \ref{stableset}. Fix $i$ and define $U_{\delta/2}(B_i)$ the $\delta/2$ open neighborhood of $B_i$ on $\Omega(f)$. If $p\in U_{\delta/2}(B_i)\cap Per(f)$, then there is $y\in B_i$ such that $d(y,p)<\delta/2$. Let $\lambda>0$ be given. We will prove that there is a $\lambda$-pseudo orbit from $p$ to $y$. By Theorem 4.1, the system satisfies periodic shadowing property, hence $\overline{Per(f)}=CR(f)$. So there is a periodic point $q$ with $$d(y,q)<\gamma=\frac{1}{2}\min\{\delta,\lambda\}.$$ Define a $\gamma$-pseudo orbit $\{x_i\}_{i\in\mathbb{Z}}$ by $x_i=f^i(p)$, if $i<0$, and $x_i=f^i(q)$, if $i\geq 0$. By shadowing property and Theorem $5.1$ there is
$$x\in W^u_\epsilon(p)\cap W^s_\epsilon(q)\subset W^u(p)\cap W^s(q). $$
Hence, if $f^M(p)=p$ and $f^N(q)=q$ there is $k>0$ such that 
$$d(f^{-kM+1}(p),f^{-kM+1}(x))<\lambda$$ and $$d(f^{kN}(q),f^{kN}(x))<\lambda/2.$$ 

Take $\{y_i\}_{i=0}^{k(M+N)}$, defined by $y_0=p=f^{-kM}(p)$, $y_i=f^{-kM+i}(x)$ if $1\leq i< kM+kN$, and $y_{k(M+N)}=y$. Then 
$$d(f(y_0),y_1)=d(f^{-kM+1}(p),f^{-kM+1}(x))<\lambda$$   
and
\begin{eqnarray*} 
d(f(y_{k(M+N)-1}),y_{k(M+N)})&=&d(f^{kN}(x),y) \\
&\leq&d(f^{kN}(x),f^{kN}(q))+d(f^{kN}(q),y) \\
&=&d(f^{kN}(x),f^{kN}(q))+d(q,y)<\frac{\lambda}{2}+\frac{\lambda}{2}=\lambda. 
\end{eqnarray*}
Therefore, $\{y_i\}_{i=0}^{k(M+N)}$ is a $\lambda$-pseudo orbit from $p$ to $y$. With a similar argument we can find a $\lambda$-pseudo orbit from $y$ to $p$, and thus $p\stackrel{\lambda}{\rightleftharpoons} y$. Since $\lambda>0$ is arbitrary, we have $p\rightleftharpoons y$, which implies $p\in B_i.$
Hence
$$B_i\supset \overline{U_{\delta/2}(B_i)\cap Per(f)} \supset U_{\delta/2}(B_i)\cap \overline{Per(f)}=U_{\delta/2}(B_i).$$So $B_i$ is open. Since $\Omega=\bigcup_{i\in\Lambda}B_i$ is compact,  $\Lambda$ is finite.

Now, if $U$ and $V$ are open sets in $B_i$, take $x$, $y$ and $r>0$ such that $B_{r}(x)\subset U$ and $B_{r}(y)\subset V$. Let $\delta_2>0$ be given by the shadowing property for $r$. Since $x\rightleftharpoons y$ there is  a finite sequence $x=z_1,z_2,...,z_l=y$ of points such that $d(f(z_i),z_{i+1})<\delta_2$, $0\leq i\leq l-1$. By shadowing there is $z\in B_r(x)$ such that $f^l(z)\in B_r(y)$, hence $f^l(U)\cap V\neq\emptyset$. Then $f|B_i$ is transitive, which finishes the proof of the first part.

To check the second part, fix $B=B_k$, and take a periodic point $p\in B$ with period $m$ and define $C_p=\overline{W^s(p)\cap B}$. Let $\delta>0$ be given by the shadowing property for $\epsilon=\frac{1}{2}\epsilon_1>0$, where $\epsilon_1>0$ is given by Theorem \ref{stableset}. We assume that $\delta<\epsilon/2$.
\begin{enumerate}
\item $C_p$ is open: Take a periodic point $q\in U_{\delta/4}(C_p)\cap B$ with period $n$. So there is $x\in W^s(p)\cap B$ such that $d(x,q)<\delta/2$. Define a $\delta$-pseudo orbit $\{y_i\}_{i\in\mathbb{Z}}$ by $y_i=f^i(q)$ if $i<0$, and $y_i=f^i(x)$ if $i\geq 0$. By shadowing there is $y\in W^u_{\epsilon/2}(q)\cap W^s_{\epsilon/2}(x)$. There is $N\in\mathbb{N}$ such that if $i\geq N$, then $d(f^i(x),f^i(p))<\epsilon/2$. Therefore if $mi>N$ we have
\begin{eqnarray*}d(f^j(f^{mi}(y)),f^j(p))&=&d(f^{j+mi}(y),f^{j+mi}(p)) \\
&\leq& d(f^{j+mi}(y),f^{j+mi}(x))+d(f^{j+mi}(x),f^{j+mi}(p)) \\
&<&\frac{\epsilon}{2}+\frac{\epsilon}{2}=\epsilon.
\end{eqnarray*}
Since $p$ and $q$ are periodic points, we have $f^{mi}(y)\in W^s_\epsilon(p)\subset W^s(p)$ and $y\in W^u_\epsilon(q)\subset W^u(q)$. Therefore, for every $k\in\mathbb{N}$, we have $f^{-m(nk)}(y)\in W^s(p)$. Since $y\in W^u_\epsilon(q)$ and $q$ has period $n$, we have $$d(f^{-m(nk)}(y),q)=d(f^{-m(nk)}(y),f^{-m(nk)}(q))\rightarrow 0.$$
Hence $q\in C_p$ and this shows that $C_p$ is open.
\item If $q$ is a periodic point such that $q\in C_p$, then $C_q\subset C_p$: Assume that the period of $p$ is $m$ and the period of $q$ is $n$. Since $q\in C_p$ there is a sequence $\{y_i\}_{i\in\mathbb{N}}$ such that $y_i\in W^s(p)\cap B$ and $y_n\rightarrow q$. Fix $x\in W^s(q)\cap B$ and let $M\in\mathbb{N}$ be such that if $i\geq M$, then $d(q,y_i)<\epsilon$. For $i\geq M$ define $\epsilon_i=\epsilon/(i+1)$ and let $\delta_i>0$ be given by the shadowing property for $\epsilon_i$. Take $L_i\in\mathbb{N}$ such that $d(y_{L_i},q)<\delta_i/2$ and $$d(f^{n(mL_i)}(x),q)<\frac{\delta_i}{2}.$$
Define a $\delta_i$-pseudo orbit by $z_j=f^j(x)$ if $j<nmL_i$, and $z_j=f^{j-nmL_i}(y_{L_i})$ if $j\geq nmL_i$. Therefore there is $x_i\in B$ such that $$d(f^j(x_i),z_j)<\epsilon_i,$$ for every $j\in\mathbb{Z}$. Thus
$d(x_i,x)=d(x_i,z_0)<\epsilon_i$,
and $x_i\rightarrow x$, when $i\rightarrow\infty$. On the other hand, we have $x_i\in W^s(p)$. Indeed, if $i\in\mathbb{N}$ is fixed, choose $N\in\mathbb{N}$ such that if $j\geq N$, then $d(f^j(y_i),f^j(p))<\epsilon/2$. Hence if $j\geq N$,
\begin{eqnarray*} d(f^j(f^{mnL_i}(x_i)),f^j(p))&\leq&d(f^j(f^{mnL_i}(x_i)),f^j(y_i))+d(f^j(y_i),f^j(p)) \\
&=&d(f^{j+mnL_i}(x_i),z_{j+mnL_i})+d(f^j(y_i),f^j(p)) \\
&<&\epsilon_i+\frac{\epsilon}{2}\leq\epsilon,
\end{eqnarray*}
and thus $f^{j+mnL_i}(x_i)\in W^s_\epsilon(f^N(p))$. Since $p$ is periodic with period $m$, by the choice of $\epsilon$ we have $f^{j+mnL_i}(x_i)\in W^s(f^N(p))$, hence $x_i\in W^s(f^{-m(nL_i)}(p))=W^s(p)$ and $x\in C_p$.
\item If $q$ is a periodic point such that $q\in C_p$, then $C_p=C_q$: By the previous claim 2. it suffices to prove that $p\in C_q$. If this is not the case we let
$$d=\mbox{\rm dist}(K,C_q),$$
where $K=C_p\setminus C_q$. Since $C_p$ and $C_q$ are compact and open sets, $K$ is compact and $d>0$. Since $q\in C_p$ there is $x\in W^s(p)$ such that $$d(x,q)<d/2.$$
Then $x\in C_q$ and therefore $f^{nmk}(x)\in C_q$, for every $k\in\mathbb{N}$. But
$$d(f^{nmk}(x),p)=d(f^{nmk}(x),f^{nmk}(p))\rightarrow 0.$$
Since $C_q$ is closed $p\in C_q$. This contradiction finishes the proof.
\item If $p$ and $q$ are periodic points with $C_p\cap C_q\neq\emptyset$, then $C_p=C_q$: Fix $x\in C_p\cap C_q$. Let $\delta>0$ be given by the shadowing property for $\epsilon/2$. Since $x\in B\subset \overline{Per(f)}$ there is a periodic point $z$ with period $l$ such that $d(x,z)<\delta/2$, and since $x\in C_p$ there is $y\in W^s(p)$ such that $d(x,y)<\delta/2$. Define a $\delta$-pseudo orbit $x_i=f^i(z)$ if $i<0$, and $x_i=f^i(y)$ if $i\geq 0$. By shadowing there is $w$ such that $$d(f^i(w),x_i)<\epsilon/2,$$ for every $i\in\mathbb{Z}$. Since $y\in W^s(p)$ there is $N\in\mathbb{N}$ such that if $i\geq mN$, then $d(f^i(w),f^i(p))<\epsilon$, and thus $f^{Nm}(w)\in W^s_\epsilon(p)\subset W^s(p)$.  Since $m$ is the period of $p$, we have $f^{-(kl)m}(w)\in W^s(p)$, for every $k\in\mathbb{N}$. On the other hand,
\begin{eqnarray*} d(f^{-i}(w),f^{-i}z)&=&d(f^{-i}(w),x_{-i})<\frac{\epsilon}{2},
\end{eqnarray*}
for every $i\in\mathbb{N}$. Thus $w\in W^u_\epsilon(z)\subset W^u(z)$. Then
\begin{eqnarray*} d(f^{-klm}(w),z)&=&d(f^{-klm}(w),f^{-klm}(z))\rightarrow 0.
\end{eqnarray*}
Hence $z\in C_p$. By the previous claim 3, it follows that $C_z=C_p$. Replacing $p$ by $q$ in the foregoing argument shows that $C_z=C_q$. Therefore $C_p=C_z=C_q$.
\item There is $M\in\mathbb{N}$ such that $B=\bigcup_{i=0}^{M-1}C_{f^{i}(p)}$, and $f^M|_{C_p}$ is topological mixing: Since the period of $p$ is equal $m$ there is a smallest $M\leq m$ such that $C_{f^M(f^i(p))}=C_{f^i(p)}$, for every $i\in\mathbb{N}$. Clearly $\bigcup_{i=0}^{M-1}C_{f^{i}(p)}\subset B$. Let $x\in B\setminus \bigcup_{i=0}^{M-1}C_{f^{i}(p)}$. Since $B$ is open and $\bigcup_{i=0}^{M-1}C_{f^{i}(p)}$ is closed, there is an open set $U\subset B$ such that $x\in U$ and $U\cap \bigcup_{i=0}^{M-1}C_{f^{i}(p)}=\emptyset$. Since $f|_B$ is topologically transitive there is $z\in C_p$ and $l\in\mathbb{N}$ such that $f^l(z)\in U$. This contradiction shows that $B=\bigcup_{i=0}^{M-1}C_{f^{i}(p)}$. Now, let $U$ and $V$ open sets in $C_p$, and choose  a periodic point $q\in V$ with period $n$. Therefore $n=Mr$ for some $r\in\mathbb{N}$ and there is $\delta>0$ such that $B_\delta(q)\subset V$. Since $C_{f^j(q)}=C_{f^j(p)}$, for every $0\leq j<n$ there is a point $x_j\in U\cap W^s(f^{Mj}(q))$. Hence there is $N_j\in\mathbb{N}$ such that if $i\geq N_j$, then
$$d(f^i(x_j),f^i(f^{Mj}(q)))<\delta.$$
Thus if $i>0$ is such that $Mni+Mn-Mj\geq N_j$, then
\begin{eqnarray*}d(f^{Mni+Mn-Mj}(x_j),q)&=&d(f^{Mni+Mn-Mj}(x_j),f^{Mni+Mn}(q)) \\
&=&d(f^{Mni+Mn-Mj}(x_j),f^{Mni+Mn-Mj}(f^{Mj}(q)))\\
&<&\delta.
\end{eqnarray*}
Define, $N$ to be the least integer greater or equal to
$$\max\left\{\frac{N_j}{M}+j;0\leq j<n\right\}.$$
Then for $k =n(i+1)-j\ge N$, we have  $kM= Mni+Mn-Mj \ge N_j$ and hence
$$d(f^{kM}(x_j),q)=d(f^{Mni+Mn-Mj}(x_j),q)<\delta.$$
Hence $f^{kM}(U)\cap V\neq\emptyset$, for every $k\geq N$. This shows that $f^M|_{C_p}$ is topologically mixing.
 \end{enumerate}
\end{proof}

\textbf{Acknowledgments:} The research of W.C.  was supported by CAPES and CNPq, and of M.D. by Johann-Gottfried-Herder Foundation.

\end{document}